\documentclass[12pt]{amsart}
\usepackage{hyperref}
\usepackage[cp1251]{inputenc}
\usepackage[T2A]{fontenc}
\usepackage[english]{babel}
\usepackage{amsmath,amsfonts,amssymb}
\usepackage{geometry}

\newtheorem{lemma}{Lemma}[section]
\newtheorem{proposition}{Proposition}[section]

\theoremstyle{remark}
\newtheorem{remark}{Remark}[section]

\theoremstyle{definition}

\numberwithin{equation}{section}

\begin{document}

\title[Parabolic problems and interpolation with a function parameter]
{Parabolic problems and\\interpolation with a function parameter}

\author[V. Los]{Valerii Los}

\address{Department of Higher and Applied Mathematics,
Chernigiv State Technological University, 95 Shevchenka,
Chernigiv, 14027, Ukraine}

\email{v$\underline{\phantom{k}}$\,los@yahoo.com}


\author[A. Murach]{Aleksandr A. Murach}

\address{Institute of Mathematics, National Academy of Sciences of Ukraine,
3 Tereshchenkivs'ka, Kyiv, 01601, Ukraine}

\email{murach@imath.kiev.ua}

\thanks{This research was partly supported by grant no. 01/01.12 of
National Academy of Sciences of Uk\-raine (under the joint Ukrainian--Russian
project of NAS of Ukraine and Russian Foundation of Basic Research).}




\begin{abstract}
We give an application of interpolation with a function parameter to parabolic
differential operators. We introduce the refined anisotropic Sobolev scale that
consists of some Hilbert function spaces of generalized smoothness. The latter is
characterized by a real number and a function varying slowly at infinity in
Karamata's sense. This scale is connected with anisotropic Sobolev spaces by means
of interpolation with a function parameter. We investigate a general initial--boundary
value parabolic problem in the refined Sobolev scale. We prove that the operator
corresponding to this problem sets isomorphisms between appropriate spaces
pertaining to this scale.
\end{abstract}

\maketitle

\section{Introduction}\label{sec1}

In the theory of partial differential equations, the question about regularity
properties of solutions to equations is of great importance. As a rule, an answer to
this question is given in the form of sufficient conditions for the solutions to
belong to certain function spaces. The latter depend on a finite collection of
number parameters and form a scale of spaces. The more finely the scale is calibrated by
these parameters, the more precise and complete an information about the solution
properties will be.

Basically, researchers use the two scales formed by Sobolev spaces and
H\"older--Zygmund spaces respectively. For these scales, we have the theory of
general elliptic boundary--value problems \cite{AgmonDouglisNirenberg59, Agmon65,
Hormander63, LionsMagenes72i, Roitberg96, Triebel95, WlokaRowleyLawruk95} and
parabolic initial--boundary value problems \cite{Eidelman69, EidelmanZhitarashu98,
Friedman64, LadyzhenskajaSolonnikovUraltzeva67, LionsMagenes72ii} (also see the
surveys \cite{Agranovich97, Eidelman94} and the bibliography given therein).
However, these scales proved to be coarse for some applications to differential
operators \cite{Hormander63, Hormander83, MikhailetsMurach10, 11UMJ11, 11MFAT4}.

In this connection, of interest are spaces for which a function parameter, not a
number, serves as an smoothness index. They are called spaces of generalized
smoothness. Important classes of such spaces were introduced and investigated by
L.~H\"ormander \cite{Hormander63} and L.~R.~Volevich, B.~P.~Paneah
\cite{VolevichPaneah65}. L.~H\"ormander \cite{Hormander63, Hormander83} gave a
systematical application of these spaces to the research on regularity
properties of solutions to hypoelliptic equations. Nowadays spaces of generalized
smoothness are used in various investigations \cite{Jacob010205, NicolaRodino10,
Paneah00, Triebel01}.

As regards applications~-- specifically, to the spectral theory of
differential equations~-- scales of Hilbert function spaces are
especially important. Until recently only Hilbert Sobolev scale
and its various weighted or anisotropic modifications have been
used in the theory of differential equations. Lately
V.~A.~Mikhai\-lets and the second author [18--24, 27--29]
elaborated the theory of general elliptic differential operators
and elliptic boundary--value problems in the Hilbert scales formed
by the H\"ormander spaces
\begin{equation}\label{f1.1}
H^{s,\varphi}:=B_{2,\mu}\quad\mbox{for}\quad
\mu(\xi):=(1+|\xi|^{2})^{s/2}\varphi\bigl((1+|\xi|^{2})^{1/2}\bigr).
\end{equation}
Here the number parameter $s$ is real, whereas the function
parameter $\varphi$ varies slowly at infinity in J.~Karamata's
sense. For example, $\varphi$ can be logarithmic function, its
iterations, each of their powers, and multiplications of these
functions. The class of the spaces \eqref{f1.1} contains the
Sobolev scale $\{H^{s}\}=\{H^{s,1}\}$ and is attached to it by
means of $s$ but is calibrated more finely than the Sobolev scale.
The number parameter $s$ sets the main (power) smoothness, while
the function parameter $\varphi$ defines a supplementary
(subpower) smoothness. The latter may give the broader or narrower
space $H^{s,\varphi}$ as compared with~$H^{s}$.

The spaces \eqref{f1.1} form the refined Sobolev scale. It possesses an important
interpolation property. Namely, each space \eqref{f1.1} can be obtained by
interpolation, with an appropriate function parameter, of a certain couple of
Sobolev spaces (see \cite[Sec. 3.2]{08MFAT1} or \cite[Sec.
1.3.4]{MikhailetsMurach10}). This parameter is a function that varies slowly of an
index $\theta\in(0,1)$ at infinity, in the sense of J.~Karamata. The refined Sobolev
scale is closed with respect to the interpolation with these function parameters.

For linear operators, their boundedness and Fredholm property will
be preserved when the interpolation of the corresponding spaces is
done. This fact allowed the authors by \cite{MikhailetsMurach10}
to transfer, to the full extent, the classical (Sobolev) theory of
elliptic partial differential equations to the case of the refined
Sobolev scale. As an application of this theory, we mention the
theorems on convergence almost everywhere and uniformly of
spectral expansions in eigenfunctions of self-adjoint positive
elliptic differential operators (see
\cite[Sec.~2.3]{MikhailetsMurach10} or \cite[Sec.~7.2]{12BJMA2}).
It is essential for this theorems that the smoothness index is a
function parameter (also see \cite{11UMJ11, 11MFAT4}).

Note that, the interpolation with a power parameter $\varphi(t)\equiv t^{\theta}$
plays an important role in the Sobolev theory of partial differential equations, the
exponent $\theta$ serving as a number parameter of the interpolation. A systematic
application of this interpolation to various classes of differential operators is
given by J.-L.~Lions, E.~Magenes \cite{LionsMagenes72i, LionsMagenes72ii}, and
H.~Triebel \cite{Triebel83, Triebel95}.

In this paper we give an application of interpolation with a
function parameter to parabolic partial differential equations.
They differ from elliptic equations in disparity of independent
variables (temporal and spatial), which implies the need to use
anisotropic function spaces. Therefore we introduce a certain
anisotropic analog of the refined Sobolev scale. For this analog,
we establish a theorem on the isomorphisms that are realized by
the operator corresponding to an initial--boundary value problem
for a pa\-ra\-bo\-lic equation of an arbitrary even order. This
theorem will be proved by means of the interpolation with a
function parameter between anisotropic Sobolev spaces. We use
regularly varying functions as interpolation parameters. In
order that our reasoning should be more transparent, we restrict
ourselves to the two-dimensional case and assume that the initial
conditions are homogeneous.

The paper consists of six sections. Section~\ref{sec1} is Introduction. In
Section~\ref{sec2}, we state an initial-boundary value problem for a general
parabolic equation given in a rectangular planar domain. Here we also formulate the
main result of the paper, the theorem on isomorphisms. In Section~\ref{sec3}, we
introduce and discuss the refined anisotropic Sobolev scale over $\mathbb{R}^{2}$
and its analogs for the rectangular domain. These analogs conform to the parabolic
problem under consideration. In Section~\ref{sec4}, we give necessary facts about
the interpolation with a function parameter between general Hilbert spaces. Main
Theorem is proved in Section~\ref{sec5}. Here we previously deduce necessary
interpolation formulas, which connect the introduced scale with anisotropic Sobolev
spaces. In the last section, \ref{sec6}, we indicate some applications and
generalizations of Main Theorem.

\section{Statement of the problem and main result}\label{sec2}

Let $\Omega:=(0,l)\times(0,\tau)$, where positive numbers $l$ and $\tau$ are chosen
arbitrarily. Consider the following linear parabolic initial--boundary value problem
in the open rectangle~$\Omega$:
\begin{align}\notag
&A(x,t,D_x,\partial_t)u(x,t)\\
&\equiv\sum_{\alpha+2b\beta\leq 2m}a^{\alpha,\beta}(x,t)\,D^\alpha_x\partial^\beta_t
u(x,t)=f(x,t)\quad\text{in}\quad\Omega,\label{f2.1}\\
&B_{j,0}(t,D_x,\partial_t)u(x,t)\big|_{x=0}\notag\\
&\equiv\sum_{\alpha+2b\beta\leq m_j}
b_{j,0}^{\alpha,\beta}(t)\,D^\alpha_x\partial^\beta_t u(x,t)\big|_{x=0}=g_{j,0}(t)
\quad\mbox{and}\label{f2.2}\\
&B_{j,1}(t,D_x,\partial_t)u(x,t)\big|_{x=l}\notag\\
&\equiv\sum_{\alpha+2b\beta\leq m_j}
b_{j,1}^{\alpha,\beta}(t)\,D^\alpha_x\partial^\beta_t
u(x,t)\big|_{x=l}=g_{j,1}(t)\label{f2.3}\\
&\quad\;\mbox{for}\quad0<t<\tau\quad\mbox{and}\quad j=1,\dots,m,\notag\\
&\frac{\partial^k u(x,t)}{\partial
t^k}\bigg|_{t=0}=0\quad\mbox{for}\quad0<x<l\quad\mbox{and}\quad
k=0,\ldots,\varkappa-1.\label{f2.4}
\end{align}

Here $b$, $m$, and all $m_j$ are arbitrarily fixed integers such that $m\geq
b\geq1$, $\varkappa:=m/b\in\mathbb{Z}$, and $m_j\geq0$. All coefficients of the
partial differential expressions $A:=A(x,t,D_x,\partial_t)$ and
$B_{j,k}:=B_{j,k}(t,D_x,\partial_t)$, with $j\in\{1,\dots,m\}$ and $k\in\{0,\,1\}$,
are supposed to be complex-valued and infinitely smooth functions; namely,
$a^{\alpha,\beta}\in C^{\infty}(\overline{\Omega})$ and $b_{j,k}^{\alpha,\beta}\in
C^{\infty}[0,\tau]$, where $\overline{\Omega}:=[0,l]\times[0,\tau]$ as usual. We use
the notation $D_x:=i\,\partial/\partial x$ and $\partial_t:=\partial/\partial t$ for
partial derivatives and take summation over the integer-valued indexes
$\alpha,\beta\geq0$ satisfying the conditions indicated.

Recall \cite[\S~9, Subsec.~1]{AgranovichVishik64} that the initial--boundary value
problem \eqref{f2.1}--\eqref{f2.4} is said to be parabolic in $\Omega$ if the
following three conditions are fulfilled:
\begin{itemize}
\item [(i)] Given any $x\in[0,l]$, $t\in[0,\tau]$, $\xi\in\mathbb{R}$, and
$p\in\mathbb{C}$ with $\mathrm{Re}\,p\geq0$, we have
\begin{align*}
&A^{(0)}(x,t,\xi,p)\\
&\equiv\sum_{\alpha+2b\beta=2m} a^{\alpha,\beta}(x,t)\,\xi^\alpha
p^{\beta}\neq0\quad \mbox{whenever}\quad|\xi|+|p|\neq0.
\end{align*}
\item [(ii)] Let $x\in\{0,l\}$, $t\in[0,\tau]$, and $p\in\mathbb{C}\setminus\{0\}$ with
$\mathrm{Re}\,p\geq0$ be arbitrary. Then the polynomial $A^{(0)}(x,t,\xi,p)$ in
$\xi\in\mathbb{C}$ has $m$ roots $\xi^{+}_{j}(x,t,p)$, $j=\nobreak1,\ldots,m$, with
positive imaginary part and $m$ roots with negative imaginary part provided that
each root is taken the number of times equal to its multiplicity.
\item [(iii)] Assume that $x$, $t$, and $p$ are the same as ones considered in (ii).
Let $k:=0$ if $x=0$, and let $k:=1$ if $x=l$. Then the polynomials
$$
B_{j,k}^{(0)}(t,\xi,p)\equiv\sum_{\alpha+2b\beta=m_{j}}
b_{j,k}^{\alpha,\beta}(t)\,\xi^{\alpha}p^{\beta},\quad j=1,\dots,m,
$$
in $\xi$ are linearly independent modulo
$$
\prod_{j=1}^{m}\bigl(\xi-\xi^{+}_{j}(x,t,p)\bigr).
$$
\end{itemize}

Consider the linear mapping
\begin{equation}\label{f2.5}
\begin{aligned}
C^{\infty}_{+}(\overline{\Omega}) & \ni u\mapsto (Au,Bu)
\\
& :=\bigl(Au,B_{1,0}u,B_{1,1}u,\ldots,B_{m,0}u,B_{m,1}u\bigr)\in
C^{\infty}_{+}(\overline{\Omega})\times\bigl(C^{\infty}_{+}[0,\tau]\bigr)^{2m},
\end{aligned}
\end{equation}

which is associated with the parabolic problem \eqref{f2.1}--\eqref{f2.4}. Here
\begin{equation*}
\begin{aligned}
C^{\infty}_{+}(\overline{\Omega}):&=\bigl\{w\!\upharpoonright\overline{\Omega}:\,
w\in C^{\infty}(\mathbb{R}^{2}),\;\,
\mathrm{supp}\,w\subseteq\mathbb{R}\times[0,\infty)\bigr\}\\
&=\bigl\{u\in
C^{\infty}(\overline{\Omega}):\,\partial_{t}^{\beta}u(x,t)|_{t=0}=0\;\,\mbox{for
all}\;\,\beta\in\mathbb{N}\cup\{0\},\;x\in[0,l]\bigr\}
\end{aligned}
\end{equation*}
and
\begin{equation*}
\begin{aligned}
C^{\infty}_{+}[0,\tau]:&=\bigl\{h\!\upharpoonright[0,\tau]:\,h\in
C^{\infty}(\mathbb{R}),\;\,\mathrm{supp}\,h\subseteq[0,\infty)\bigr\}\\
&=\bigl\{v\in C^{\infty}[0,\tau]:\,v^{(\beta)}(0)=0\;\,\mbox{for
all}\;\,\beta\in\mathbb{N}\cup\{0\}\bigr\}.
\end{aligned}
\end{equation*}
In the paper, functions (and distributions) are supposed to be complex-valued unless otherwise stated.

The mapping \eqref{f2.5} sets a one-to-one correspondence between the spaces
$C^{\infty}_{+}(\overline{\Omega})$ and
$C^{\infty}_{+}(\overline{\Omega})\times\bigl(C^{\infty}_{+}[0,\tau]\bigr)^{2m}$
(see Remark \ref{rem3.3} below). Our purpose is to show that this mapping extends by
continuity to an isomorphism between appropriate couples of Hilbert function spaces
of generalized smoothness. Namely, we will prove the following result.

Let $\sigma_0$ be the smallest integer such that
$$
\sigma_0\geq2m,\quad\sigma_0\geq m_j+1\;\;\mbox{for all}\;\;j\in\{1,\ldots,m\},
\quad\mbox{and}\quad\frac{\sigma_0}{2b}\in\mathbb{Z}.
$$
Note, if $m_j\leq2m-1$ for every $j\in\{1,\ldots,m\}$, then $\sigma_0=2m$.

\medskip

\textbf{Main Theorem.} \it Let a real number $\sigma>\sigma_0$ and function
parameter $\varphi\in\mathcal{M}$ be chosen arbitrarily. Then the mapping
\eqref{f2.5} extends uniquely (by continuity) to an isomorphism
\begin{equation}\label{f2.6}
\begin{split}
(A,B)& :\,H^{\sigma,\sigma/(2b),\varphi}_{+}(\Omega)
\\
& \leftrightarrow\;
H^{\sigma-2m,(\sigma-2m)/(2b),\varphi}_{+}(\Omega)\oplus
\bigoplus_{j=1}^{m}\bigl(H^{(\sigma-m_j-1/2)/(2b),\varphi}_{+}(0,\tau)\bigr)^{2}.
\end{split}
\end{equation} \rm

The class $\mathcal{M}$ and the Hilbert function spaces occurring in \eqref{f2.6}
will be defined in the next section. These spaces form the refined Sobolev scales.

If $\varphi\equiv1$, then the operator \eqref{f2.6} acts between Sobolev spaces. In
this case, this theorem was proved by M.~S.~Agranovich and M.~I.~Vishik
\cite[Theorem 11.1]{AgranovichVishik64} on the assumption that
$\sigma/(2b)\in\mathbb{Z}$. Their result includes the limiting case of
$\sigma=\sigma_0$ and relates to general parabolic problems with nonhomogeneous
initial conditions.

\section{Refined Sobolev scales}\label{sec3}

In this section we will introduce and discuss the function spaces used in the
statement of Main Theorem. The regularity properties of the distributions belonging
to these spaces are characterized by two number parameters and a function parameter.
The latter runs over a certain function class $\mathcal{M}$, which is defined as
follows.

The class $\mathcal{M}$ consists of all functions
$\varphi:[1,\infty)\rightarrow(0,\infty)$ such that
\begin{itemize}
\item [a)] $\varphi$ is Borel measurable on $[1,\infty)$;
\item [b)] both the functions $\varphi$ and $1/\varphi$ are bounded on each
compact interval $[1,b]$, with $1<b<\infty$;
\item [c)] $\varphi$ is a slowly varying function at infinity in the sense of
J.~Karamata; i.e.,
\begin{equation}\label{f3.1}
\lim_{r\rightarrow\infty}\frac{\varphi(\lambda r)}{\varphi(r)}=1\quad\mbox{for
every}\quad \lambda>0.
\end{equation}
\end{itemize}

\begin{remark}\label{rem3.1} \rm
The theory of slowly varying functions is set forth in the monographs
\cite{BinghamGoldieTeugels89, Seneta76}. We give an important and standard example
of functions satisfying \eqref{f3.1} if we put
\begin{equation}\label{f3.2}
\varphi(r):=(\log r)^{\theta_{1}}\,(\log\log r)^{\theta_{2}} \ldots
(\,\underbrace{\log\ldots\log}_{k\;\mbox{\small{times}}}r\,)^{\theta_{k}}
\quad\mbox{for}\quad r\gg1,
\end{equation}
where the parameters $k\in\mathbb{N}$ and $\theta_{1},
\theta_{2},\ldots,\theta_{k}\in\mathbb{R}$ are chosen arbitrarily. The functions
\eqref{f3.2} form the logarithmic multiscale, which has a number of applications in
the theory of function spaces. Some other examples of slowly varying functions can
be found in \cite[Sec. 1.3.3]{BinghamGoldieTeugels89} and \cite[Sec.
1.2.1]{MikhailetsMurach10}.
\end{remark}

Let $s\in\mathbb{R}$, $\varphi\in\mathcal{M}$, and $\gamma:=1/(2b)$. By definition,
the linear space $H^{s,s\gamma,\varphi}(\mathbb{R}^{2})$ consists of all tempered
distributions $w\in \mathcal{S}'(\mathbb{R}^{2})$ such that their Fourier transform
$\widetilde{w}$ (in two variables) is locally Lebesgue integrable over
$\mathbb{R}^{2}$ and satisfies the condition
\begin{equation}\label{f3.3}
\int\limits_{-\infty}^{\infty}\,\int\limits_{-\infty}^{\infty}
r_{\gamma}^{2s}(\xi,\eta)\,\varphi^{2}(r_{\gamma}(\xi,\eta))\,
|\widetilde{w}(\xi,\eta)|^{2}\,d\xi d\eta<\infty.
\end{equation}
Here and below we use the notation
$$
r_{\gamma}(\xi,\eta):=\bigl(1+|\xi|^2+|\eta|^{2\gamma}\bigr)^{1/2}\quad\mbox{for
each}\quad\xi,\eta\in\mathbb{R}.
$$
The space $H^{s,s\gamma,\varphi}(\mathbb{R}^{2})$ is endowed with the inner product
$$
(w_{1},w_{2})_{H^{s,s\gamma,\varphi}(\mathbb{R}^{2})}:=
\int\limits_{-\infty}^{\infty}\,\int\limits_{-\infty}^{\infty}
r_{\gamma}^{2s}(\xi,\eta)\,\varphi^{2}(r_{\gamma}(\xi,\eta))\,
\widetilde{w_{1}}(\xi,\eta)\,\overline{\widetilde{w_{2}}(\xi,\eta)}\,d\xi d\eta,
$$
where $w_{1},w_{2}\in H^{s,s\gamma,\varphi}(\mathbb{R}^{2})$. It induces the norm
$$
\|w\|_{H^{s,s\gamma,\varphi}(\mathbb{R}^{2})}:=
(w,w)_{H^{s,s\gamma,\varphi}(\mathbb{R}^{2})}^{1/2},
$$
which is equal to the square root of the left-hand side of inequality \eqref{f3.3}.

Note that $H^{s,s\gamma,\varphi}(\mathbb{R}^{2})$ is the inner
product H\"ormander space $\mathcal{B}_{2,\mu}(\mathbb{R}^{2})$
which corres\-ponds to the function parameter $$
\mu(\xi,\eta):=r_{\gamma}^{s}(\xi,\eta)\,\varphi(r_{\gamma}(\xi,\eta))\quad
\mbox{for}\quad\xi,\eta\in\mathbb{R}. $$ We refer the reader to
the monographs by L.~H\"ormander \cite[Sec.~2.2]{Hormander63},
\cite[Sec.~10.1]{Hormander83}, and to the paper by L.~R.~Volevich
and B.~P.~Paneah \cite{VolevichPaneah65}, where such spaces are
investigated systematically. It follows from properties of
H\"ormander spaces that the space
$H^{s,s\gamma,\varphi}(\mathbb{R}^{2})$ is Hilbert and separable,
is embedded continuously in $\mathcal{S}'(\mathbb{R}^{2})$, and
that
the set $C^{\infty}_{0}(\mathbb{R}^{2})$ is dense in
$H^{s,s\gamma,\varphi}(\mathbb{R}^{2})$.

\begin{remark}\label{rem3.2} \rm
We use conventional notation for main function spaces. So,
$\mathcal{S}'(\mathbb{R}^{n})$ denotes the linear topological L.~Schwartz space of
all tempered distributions given in $\mathbb{R}^{n}$, with $n\in\mathbb{N}$. If $G$
is an open subset of $\mathbb{R}^{n}$ (in particular, $G=\mathbb{R}^{n}$), then
$C^{\infty}_{0}(G)$ stands for the class of all functions $w\in
C^{\infty}(\mathbb{R}^{n})$ such that their support is a compact subset of $G$. We
may naturally identify a function $w\in C^{\infty}_{0}(G)$ with its restriction to
$G$; from the context it will always be understood on which set~--- $\mathbb{R}^{n}$
or $G$~--- the function $w$ is considered. The designation $L_{2}(G,d\mu)$ refers to
the Hilbert space of all functions that are square integrable over $G$ with respect
to a Radon measure~$\mu$. Specifically, if $\mu$ is the Lebesgue measure, then we
omit $d\mu$ and write $L_{2}(G)$.
\end{remark}

If $\varphi(r)\equiv1$, then $H^{s,s\gamma,\varphi}(\mathbb{R}^{2})$ becomes the
anisotropic Sobolev space of order $(s,s\gamma)$; we denote this space by
$H^{s,s\gamma}(\mathbb{R}^{2})$. Note that, in the case where
$s,s\gamma\in\mathbb{N}$ the space $H^{s,s\gamma}(\mathbb{R}^{2})$ consists of all
functions $w(x,t)$ such that $w$, $D_{x}^{s}w$, and $\partial_{t}^{s\gamma}w$ are
square integrable over $\mathbb{R}^{2}$, providing the partial derivatives are
understood in the sense of the theory of distributions. In this case, we have the
equivalence of Hilbert norms
\begin{equation}\label{f3.4}
\|w\|_{H^{s,s\gamma}(\mathbb{R}^{2})}
\asymp\Biggl(\:\int\limits_{-\infty}^{\infty}\int\limits_{-\infty}^{\infty}
\bigl(|w(x,t)|^{2}+|D_{x}^{s}w(x,t)|^{2}+|\partial_{t}^{s\gamma}w(x,t)|^{2}\bigr)
dxdt\Biggr)^{1/2}.
\end{equation}

Every space $H^{s,s\gamma,\varphi}(\mathbb{R}^{2})$, with $s\in\mathbb{R}$ and
$\varphi\in\mathcal{M}$, is closely connected to aniso\-tropic Sobolev spaces.
Specifically, we have the continuous and dense embeddings
\begin{equation}\label{f3.5}
H^{s_{1},s_{1}\gamma}(\mathbb{R}^{2})\hookrightarrow
H^{s,s\gamma,\varphi}(\mathbb{R}^{2})\hookrightarrow
H^{s_{0},s_{0}\gamma}(\mathbb{R}^{2})\quad\mbox{whenever}\quad s_{0}<s<s_{1}.
\end{equation}
They follow from the next property of $\varphi\in\mathcal{M}$: for each
$\varepsilon>0$ there exists a number $c=c({\varepsilon})\geq1$ such that
$c^{-1}r^{-\varepsilon}\leq\varphi(r)\leq c\,r^{\varepsilon}$ for all $r\geq1$ (see
\cite[Sec.~1.5, Subsec.~1]{Seneta76}).

Consider the class of Hilbert function spaces
\begin{equation}\label{f3.6}
\bigl\{H^{s,s\gamma,\varphi}(\mathbb{R}^{2}):\,
s\in\mathbb{R},\,\varphi\in\mathcal{M}\,\bigr\}.
\end{equation}
Owing to the embeddings \eqref{f3.5}, we may assert that in \eqref{f3.6} the function
parameter $\varphi$ defines a supplementary (subpower) smoothness with respect to
the basic (power) anisotropic $(s,s\gamma)$-smoothness. Specifically, if
$\varphi(r)\rightarrow\infty$ [$\varphi(r)\rightarrow0$] as $r\rightarrow\infty$,
then $\varphi$ defines a positive [negative] supplementary smoothness. In other
words, $\varphi$ refines the power smoothness $(s,s\gamma)$.

Therefore we will naturally call \eqref{f3.6} the refined anisotropic Sobolev scale
over $\mathbb{R}^{2}$; here $\gamma$ serves as an anisotropy parameter.

Using this scale, let us introduce some function spaces related to the parabolic
problem under consideration. As before, $s\in\mathbb{R}$ and
$\varphi\in\mathcal{M}$. We put
$$
H^{s,s\gamma,\varphi}_{+}(\mathbb{R}^{2}):=\bigl\{w\in
H^{s,s\gamma,\varphi}(\mathbb{R}^{2}):\,
\mathrm{supp}\,w\subseteq\mathbb{R}\times[0,\infty)\bigr\}.
$$
The linear space $H^{s,s\gamma,\varphi}_{+}(\mathbb{R}^{2})$ is endowed with the
inner product and norm in $H^{s,s\gamma,\varphi}(\mathbb{R}^{2})$. The space
$H^{s,s\gamma,\varphi}_{+}(\mathbb{R}^{2})$ is complete (Hilbert) because of the
continuous embedding
$$
H^{s,s\gamma,\varphi}(\mathbb{R}^{2})\hookrightarrow\mathcal{S}'(\mathbb{R}^{2}).
$$

Next, we define the normed linear space
\begin{equation}\label{f3.7}
\begin{gathered}
H^{s,s\gamma,\varphi}_{+}(\Omega):=\bigl\{w\!\upharpoonright\Omega:\,w\in
H^{s,s\gamma,\varphi}_{+}(\mathbb{R}^{2})\bigr\},\\
\|u\|_{H^{s,s\gamma,\varphi}_{+}(\Omega)}:=
\inf\bigl\{\|w\|_{H^{s,s\gamma,\varphi}(\mathbb{R}^{2})}:\,w\in
H^{s,s\gamma,\varphi}_{+}(\mathbb{R}^{2}),\;\;w=u\;\,\mbox{in}\;\,\Omega\bigr\},
\end{gathered}
\end{equation}
with $u\in H^{s,s\gamma,\varphi}_{+}(\Omega)$. In other words,
$H^{s,s\gamma,\varphi}_{+}(\Omega)$ is the factor space of the space
$H^{s,s\gamma,\varphi}_{+}(\mathbb{R}^{2})$ by its subspace
\begin{equation}\label{f3.8}
H^{s,s\gamma,\varphi}_{Q}(\mathbb{R}^{2}):=\bigl\{w\in
H^{s,s\gamma,\varphi}(\mathbb{R}^{2}):\, \mathrm{supp}\,w\subseteq
Q:=\mathbb{R}\times[0,\infty)\setminus\Omega\bigr\}.
\end{equation}
Hence, the space $H^{s,s\gamma,\varphi}_{+}(\Omega)$ is Hilbert. The norm
\eqref{f3.7} is induced by the inner product
$$
(u_{1},u_{2})_{H^{s,s\gamma,\varphi}_{+}(\Omega)}:= (w_{1}-\Upsilon
w_{1},w_{2}-\Upsilon w_{2})_{H^{s,s\gamma,\varphi}(\mathbb{R}^{2})},
$$
where $w_{j}\in H^{s,s\gamma,\varphi}(\mathbb{R}^{2})$, $w_{j}=u_{j}$ in $\Omega$
for each $j\in\{1,\,2\}$, and $\Upsilon$ is the orthogonal projector of the space
$H^{s,s\gamma,\varphi}_{+}(\mathbb{R}^{2})$ onto its subspace \eqref{f3.8}.

Note that both Hilbert spaces $H^{s,s\gamma,\varphi}_{+}(\mathbb{R}^{2})$ and
$H^{s,s\gamma,\varphi}_{+}(\Omega)$ are separable. The set
$C^{\infty}_{0}(\mathbb{R}\times(0,\infty))$ is dense in
$H^{s,s\gamma,\varphi}_{+}(\mathbb{R}^{2})$ \cite[Lemma 3.3]{VolevichPaneah65}; this
implies the density of $C^{\infty}_{+}(\overline{\Omega})$ in
$H^{s,s\gamma,\varphi}_{+}(\Omega)$.

It remains to introduce the function spaces in which the right-hand sides of the
boundary-value conditions \eqref{f2.2} and \eqref{f2.3} are considered. Let
$s\in\mathbb{R}$ and $\varphi\in\mathcal{M}$. By definition, the linear space
$H^{s,\varphi}(\mathbb{R})$ consists of all tempered distributions
$h\in\mathcal{S}'(\mathbb{R})$ such that their Fourier transform $\widehat{h}$ is
locally Lebesgue integrable over $\mathbb{R}$ and satisfies the condition
$$
\int\limits_{-\infty}^{\infty}\langle\xi\rangle^{2s}\,\varphi^{2}(\langle\xi\rangle)\,
|\widehat{h}(\xi)|^{2}\,d\xi<\infty.
$$
Here, as usual, $\langle\xi\rangle:=(1+|\xi|^{2})^{1/2}$ is the smooth modulus of
$\xi\in\mathbb{R}$. The space $H^{s,\varphi}(\mathbb{R})$ is endowed with the inner
product
$$
(h_{1},h_{2})_{H^{s,\varphi}(\mathbb{R})}:=\int\limits_{-\infty}^{\infty}
\langle\xi\rangle^{2s}\,\varphi^{2}(\langle\xi\rangle)\,
\widehat{h_{1}}(\xi)\,\overline{\widehat{h_{2}}(\xi)}\,d\xi,
$$
where $h_{1},h_{2}\in H^{s,\varphi}(\mathbb{R})$. It induces the norm
$$
\|h\|_{H^{s,\varphi}(\mathbb{R})}:=(h,h)_{H^{s,\varphi}(\mathbb{R})}^{1/2}.
$$

Notice that $H^{s,\varphi}(\mathbb{R})$ is the inner product H\"ormander space
$\mathcal{B}_{2,\mu}(\mathbb{R})$ corresponding to the function parameter
$\mu(\xi):=\langle\xi\rangle^{s}\varphi(\langle\xi\rangle)$ of $\xi\in\mathbb{R}$
(see the references \cite{Hormander63, Hormander83, VolevichPaneah65} mentioned
above). Therefore $H^{s,\varphi}(\mathbb{R})$ is a separable Hilbert space embedded
continuously in $\mathcal{S}'(\mathbb{R})$, and the set $C^{\infty}_{0}(\mathbb{R})$
is dense in $H^{s,\varphi}(\mathbb{R})$.

If $\varphi(r)\equiv1$, then $H^{s,\varphi}(\mathbb{R})$ becomes the Sobolev space
$H^{s}(\mathbb{R})$ of order $s$. Analogously to \eqref{f3.5}, we have the
continuous and dense embedding
\begin{equation}\label{f3.9}
H^{s_{1}}(\mathbb{R})\hookrightarrow H^{s,\varphi}(\mathbb{R})\hookrightarrow
H^{s_{0}}(\mathbb{R})\quad\mbox{whenever}\quad
s_{0}<s<s_{1},\quad\varphi\in\mathcal{M}.
\end{equation}

The class of Hilbert function spaces
\begin{equation}\label{f3.10}
\bigl\{H^{s,\varphi}(\mathbb{R}):\,s\in\mathbb{R},\,\varphi\in\mathcal{M}\,\bigr\}
\end{equation}
is called the refined Sobolev scale over $\mathbb{R}$ (see \cite[Sec.
1.3.3]{MikhailetsMurach10} and \cite[Sec. 3.2]{12BJMA2}).

Using this scale, introduce one-dimensional analogs of the spaces considered above.
We let
$$
H^{s,\varphi}_{+}(\mathbb{R}):=\bigl\{h\in H^{s,\varphi}(\mathbb{R}):\,
\mathrm{supp}\,h\subseteq[0,\infty)\bigr\}
$$
and interpret $H^{s,\varphi}_{+}(\mathbb{R})$ as a (closed) subspace of
$H^{s,\varphi}(\mathbb{R})$. Then define the normed linear space
\begin{gather*}
H^{s,\varphi}_{+}(0,\tau):=\bigl\{h\!\upharpoonright(0,\tau):\,h\in
H^{s,\varphi}_{+}(\mathbb{R})\bigr\},\\
\|v\|_{H^{s,\varphi}_{+}(0,\tau)}:=
\inf\bigl\{\|h\|_{H^{s,\varphi}(\mathbb{R})}:\,h\in
H^{s,\varphi}_{+}(\mathbb{R}),\;\;h=v\;\,\mbox{in}\;(0,\tau)\bigr\},
\end{gather*}
with $v\in H^{s,\varphi}_{+}(0,\tau)$. This space is Hilbert as it is the factor
space of $H^{s,\varphi}_{+}(\mathbb{R})$ by
\begin{equation}\label{f3.11}
\bigl\{h\in H^{s,\varphi}(\mathbb{R}):\,
\mathrm{supp}\,h\subseteq\{0\}\cup[\tau,\infty)\bigr\}.
\end{equation}
Both Hilbert spaces $H^{s,\varphi}_{+}(\mathbb{R})$ and $H^{s,\varphi}_{+}(0,\tau)$
are separable. The set $C^{\infty}_{0}(0,\infty)$ is dense in
$H^{s,\varphi}_{+}(\mathbb{R})$ \cite[Lemma 3.3]{VolevichPaneah65} so that
$C^{\infty}_{+}[0,\tau]$ is dense in $H^{s,\varphi}_{+}(0,\tau)$.

In the Sobolev case of $\varphi\equiv1$ we will omit the index $\varphi$ in the
designations of the spaces introduced.

We finish this section with the following observation.

\begin{remark}\label{rem3.3} \rm
According to the Sobolev embedding theorem and the above-mentioned result by
Agranovich and Vishik \cite[theorem 11.1]{AgranovichVishik64}, we obtain the
equalities
\begin{gather*}
C^{\infty}_{+}(\overline{\Omega})=
\bigcap_{\substack{\sigma>\sigma_0,\\\sigma/(2b)\in\mathbb{Z}}}
H^{\sigma,\sigma/(2b)}_{+}(\Omega),\\
C^{\infty}_{+}(\overline{\Omega})\times\bigl(C^{\infty}_{+}[0,\tau]\bigr)^{2m}
=\bigcap_{\substack{\sigma>\sigma_0,\\\sigma/(2b)\in\mathbb{Z}}}
(A,B)\bigl(H^{\sigma,\sigma/(2b)}_{+}(\Omega)\bigr).
\end{gather*}
It follows from them that the mapping \eqref{f2.5} sets a one-to-one correspondence
between the spaces $C^{\infty}_{+}(\overline{\Omega})$ and
$C^{\infty}_{+}(\overline{\Omega})\times\bigl(C^{\infty}_{+}[0,\tau]\bigr)^{2m}$.
\end{remark}

\section{Abstract auxiliary results}\label{sec4}

Here we recall the definition of the interpolation with a function parameter in the
case of general Hilbert spaces and then discuss the interpolation properties which
will be used in Section~\ref{sec5}. We follow the monograph \cite[Sec.
1.1]{MikhailetsMurach10} (also see \cite[Sec.~2]{08MFAT1}). It is sufficient to
restrict ourselves to separable complex Hilbert spaces.

Let $X:=[X_{0},X_{1}]$ be an ordered couple of separable complex Hilbert spaces such
that the continuous and dense embedding $X_{1}\hookrightarrow X_{0}$ holds. This
couple is said to be admissible. For $X$ there exists an isometric isomorphism
$J:X_{1}\leftrightarrow X_{0}$ such that $J$ is a self-adjoint and positive operator
on $X_{0}$ with the domain $X_{1}$. The operator $J$ is uniquely determined by the
couple $X$ and is called the generating operator for~$X$.

Let $\psi\in\mathcal{B}$, where $\mathcal{B}$ denotes the set of all Borel
measurable functions $\psi:(0,\infty)\rightarrow(0,\infty)$ such that $\psi$ is
bounded on each compact interval $[a,b]$, with $0<a<b<\infty$, and that $1/\psi$ is
bounded on every semiaxis $[a,\infty)$, with $a>0$.

Consider the operator $\psi(J)$, which is defined (and positive) in $X_{0}$ as the
Borel function $\psi$ of $J$. Denote by $[X_{0},X_{1}]_{\psi}$ or simply by
$X_{\psi}$ the domain of the operator $\psi(J)$ endowed with the inner product
$$
(u_{1},u_{2})_{X_{\psi}}:=(\psi(J)u_{1},\psi(J)u_{2})_{X_{0}}.\quad
$$
It induces the norm $\|u\|_{X_{\psi}}:=\|\psi(J)u\|_{X_{0}}$. The space $X_{\psi}$
is Hilbert and separable.

A function $\psi\in\mathcal{B}$ is called an interpolation parameter if the
following condition is fulfilled for all admissible couples $X=[X_{0},X_{1}]$ and
$Y=[Y_{0},Y_{1}]$ of Hilbert spaces and for an arbitrary linear mapping $T$ given on
$X_{0}$: if the restriction of $T$ to $X_{j}$ is a bounded operator
$T:X_{j}\rightarrow Y_{j}$ for each $j\in\{0,1\}$, then the restriction of $T$ to
$X_{\psi}$ is also a bounded operator $T:X_{\psi}\rightarrow Y_{\psi}$.

If $\psi$ is an interpolation parameter, then we say that the Hilbert space
$X_{\psi}$ is obtained by the interpolation with the function parameter $\psi$ of
the couple $X=[X_{0},X_{1}]$ (or, in other words, between the spaces $X_{0}$ and
$X_{1}$). In this case the dense and continuous embeddings $X_{1}\hookrightarrow
X_{\psi}\hookrightarrow X_{0}$ are valid.

It is known that a function $\psi\in\mathcal{B}$ is an interpolation parameter if
and only if $\psi$ is pseudoconcave in a neighbourhood of $\infty$, i.e. there is a
concave positive function $\psi_{1}(r)$ of $r\gg1$ such that both the functions
$\psi/\psi_{1}$ and $\psi_{1}/\psi$ are bounded on some neighbourhood of $\infty$.
This criterion follows from J.~Peetre's description of all interpolation functions
for the weighted $L_{p}(\mathbb{R}^{n})$-type spaces (see \cite[Theorem
5.4.4]{BerghLefstrem76}). The corresponding proof is given in \cite[Sec.
1.1.9]{MikhailetsMurach10}.

For us, it is important the next consequence of this criterion \cite[Theorem
1.11]{MikhailetsMurach10}.

\begin{proposition}\label{prop4.1}
Suppose that a function $\psi\in\mathcal{B}$ varies regularly of index $\theta$ at
infinity, with $0<\theta<1$, i.e.
$$
\lim_{r\rightarrow\infty}\;\frac{\psi(\lambda r)}{\psi(r)}=
\lambda^{\theta}\quad\mbox{for each}\quad\lambda>0.
$$
Then $\psi$ is an interpolation parameter.
\end{proposition}

\begin{remark}\label{rem4.2} \rm
In the case of power functions this proposition leads us to the classical result by
J.-L.~Lions and S.~G.~Krein, which consists in that the function $\psi(r)\equiv
r^{\theta}$ is an interpolation parameter whenever $0<\theta<1$. Here the exponent
$\theta$ is regarded as a number parameter of the interpolation.
\end{remark}

At the end of this section we formulate two properties of the
interpolation; they will be used in our proofs. The first of them
enables us to reduce the interpolation of subspaces or factor
spaces to the interpolation of initial spaces (see
\cite[Sec.~1.1.6]{MikhailetsMurach10} and
\cite[Sec.~1.17]{Triebel95}). Note that subspaces are assumed to
be closed and that we generally consider nonorthogonal projectors
onto subspaces.

\begin{proposition}\label{prop4.3}
Let $X=[X_{0},X_{1}]$ be an admissible couple of Hilbert spaces,
and let $Y_{0}$ be a subspace of $X_{0}$. Then $Y_{1}:=X_{1}\cap
Y_{0}$ is a subspace of $X_{1}$. Suppose that there exists a
linear mapping $P:X_{0}\rightarrow X_{0}$ such that $P$ is a
projector of the space $X_{j}$ onto its subspace $Y_{j}$ for every
$j\in\{0,\,1\}$. Then the couples $[Y_{0},Y_{1}]$ and
$[X_{0}/Y_{0},X_{1}/Y_{1}]$ are admissible, and
$$
\begin{aligned}
\ [Y_{0},Y_{1}]_{\psi} & = X_{\psi}\cap Y_{0},
\\
[X_{0}/Y_{0},X_{1}/Y_{1}]_{\psi} & = X_{\psi}/(X_{\psi}\cap Y_{0})
\end{aligned}
$$
with equivalence of norms. Here $\psi\in\mathcal{B}$ is an arbitrary interpolation
parameter.
\end{proposition}

The second property reduces the interpolation of direct sums of Hilbert spaces to
the interpolation of their summands.

\begin{proposition}\label{prop4.4}
Let $[X_{0}^{(j)},X_{1}^{(j)}]$, with $j=1,\ldots,p$, be a finite collection of
admissible couples of Hilbert spaces. Then
$$
\biggl[\,\bigoplus_{j=1}^{p}X_{0}^{(j)},\,\bigoplus_{j=1}^{p}X_{1}^{(j)}\biggr]_{\psi}=\,
\bigoplus_{j=1}^{p}\bigl[X_{0}^{(j)},\,X_{1}^{(j)}\bigr]_{\psi}
$$
with equality of norms. Here $\psi\in\mathcal{B}$ is an arbitrary interpolation
parameter.
\end{proposition}

\section{Proof of the main result}\label{sec5}

We will previously prove that the spaces appearing in \eqref{f2.6} can be obtained
by the interpolation with a function parameter between certain Sobolev spaces. Using
this interpolation we will deduce Main Theorem from the above-mentioned result by
Agranovich and Vishik.

In this section we suppose that
\begin{equation}\label{f5.1}
s,s_{0},s_{1}\in\mathbb{R},\quad
s_{0}<s<s_{1},\quad\mbox{and}\quad\varphi\in\mathcal{M}.
\end{equation}
Consider the function
\begin{equation}\label{f5.2}
\psi(r):=
\begin{cases}
\;r^{(s-s_{0})/(s_{1}-s_{0})}\,\varphi(r^{1/(s_{1}-s_{0})})&\text{for}\quad r\geq1, \\
\;\varphi(1) & \text{for}\quad0<r<1.
\end{cases}
\end{equation}
This function is an interpolation parameter by Proposition~\ref{prop4.1} because
$\psi$ varies regularly of index $\theta:=(s-s_{0})/(s_{1}-s_{0})$ at infinity, with
$0<\theta<1$. We will interpolate couples of Sobolev spaces with the function
parameter $\psi$.

We begin with anisotropic spaces and prove necessary interpolation formulas for the
spaces $H^{s,s\gamma,\varphi}(\mathbb{R}^{2})$,
$H^{s,s\gamma,\varphi}_{+}(\mathbb{R}^{2})$, and $H^{s,s\gamma,\varphi}_{+}(\Omega)$
deducing each next formula from the previous one. The corresponding results will be
formulated as lemmas.

\begin{lemma}\label{lem5.1}
On the assumption \eqref{f5.1} we have
\begin{equation}\label{f5.3}
H^{s,s\gamma,\varphi}(\mathbb{R}^{2})=\bigl[H^{s_{0},s_{0}\gamma}(\mathbb{R}^{2}),
H^{s_{1},s_{1}\gamma}(\mathbb{R}^{2})\bigr]_{\psi}
\end{equation}
with equality of norms.
\end{lemma}

\begin{proof}
The couple of Sobolev spaces
$$
X:=\bigl[H^{s_{0},s_{0}\gamma}(\mathbb{R}^{2}),
H^{s_{1},s_{1}\gamma}(\mathbb{R}^{2})\bigr]
$$
is admissible in view of~\eqref{f3.5}. The generating operator for this couple is
given by the formula
$$
J:\,w\mapsto\mathcal{F}^{-1}[r_{\gamma}^{s_{1}-s_{0}}\,\mathcal{F}w\,],
\quad\mbox{with}\quad w\in H^{s_{1},s_{1}\gamma}(\mathbb{R}^{2}).
$$
This follows immediately from the definition of these spaces. Here $\mathcal{F}$
and $\mathcal{F}^{-1}$ stand for the operators of the direct and inverse Fourier
transform (in two variables) of tempered distributions given in $\mathbb{R}^{2}$.

Note that $J$ is reduced to the operator of multiplication by
$r_{\gamma}^{s_{1}-s_{0}}$ with the help of the Fourier transform considered as an
isometric isomorphism
$$
\mathcal{F}:\,H^{s_{0},s_{0}\gamma}(\mathbb{R}^{2})\leftrightarrow
L_{2}\bigl(\mathbb{R}^{2},r_{\gamma}^{2s_{0}}(\xi,\eta)d\xi d\eta\bigr).
$$
Hence $\mathcal{F}$ reduces $\psi(J)$ to the operator of multiplication by the
function
$$
\psi(r_{\gamma}^{s_{1}-s_{0}}(\xi,\eta))\equiv
r_{\gamma}^{s-s_{0}}(\xi,\eta)\,\varphi(r_{\gamma}(\xi,\eta)),
$$
in view of \eqref{f5.2}. Now for each $w\in C^{\infty}_{0}(\mathbb{R}^{2})$ we may
write the following:
\begin{align*}
\|w\|_{X_{\psi}}^{2}&=\|\psi(J)w\|_{H^{s_{0},s_{0}\gamma}(\mathbb{R}^{2})}^{2}\\
&=\int\limits_{-\infty}^{\infty}\,\int\limits_{-\infty}^{\infty}
|\psi(r_{\gamma}^{s_{1}-s_{0}}(\xi,\eta))\,(\mathcal{F}w)(\xi,\eta)|^{2}\,
r_{\gamma}^{2s_{0}}(\xi,\eta)d\xi
d\eta\\&=\|w\|_{H^{s,s\gamma,\varphi}(\mathbb{R}^{2})}.
\end{align*}
This implies the equality of spaces \eqref{f5.3} as $C^{\infty}_{0}(\mathbb{R}^{2})$
is dense in both of them. (Note that $C^{\infty}_{0}(\mathbb{R}^{2})$ is dense in
the second space denoted by $X_{\psi}$ because $C^{\infty}_{0}(\mathbb{R}^{n})$ is
dense in the space $H^{s_{1},s_{1}\gamma}(\mathbb{R}^{2})$ embedded continuously and
densely in $X_{\psi}$.)
\end{proof}

To apply this lemma to the interpolation between the subspaces
$H^{s_{0},s_{0}\gamma}_{+}(\mathbb{R}^{2})$ and
$H^{s_{1},s_{1}\gamma}_{+}(\mathbb{R}^{2})$ we need the following preparatory
result.

Let $\Pi$ be an open half-plain in $\mathbb{R}^2$ such that its boundary
$\partial\Pi$ is parallel to a certain coordinate axis. The anisotropic Sobolev
space $H^{s,s\gamma}(\Pi)$ is defined as follows
\begin{gather*}
H^{s,s\gamma}(\Pi):=\bigl\{w\!\upharpoonright\Pi:\,w\in
H^{s,s\gamma}(\mathbb{R}^{2})\bigr\},\\
\|v\|_{H^{s,s\gamma}(\Pi)}:= \inf\bigl\{\|w\|_{H^{s,s\gamma}(\mathbb{R}^{2})}:\,w\in
H^{s,s\gamma}(\mathbb{R}^{2}),\;\;w=v\;\,\mbox{in}\;\,\Pi\bigr\}.
\end{gather*}
This space is Hilbert.

\begin{lemma}\label{lem5.2}
Let numbers $k\in\mathbb{N}$ and $\varepsilon>0$ be arbitrarily chosen. There exists
a bounded linear operator $T_{\Pi}^{k,\varepsilon}:L_2(\Pi)\rightarrow
L_2(\mathbb{R}^2)$ that satisfies the following conditions:
\begin{itemize}
\item [(i)] The mapping $T_{\Pi}^{k,\varepsilon}$ is an extension operator; i.e.,
$T_{\Pi}^{k,\varepsilon}v=v$ in $\Pi$ for each $v\in L_2(\Pi)$.
\item [(ii)] If $s,s\gamma\in\mathbb{N}\cap[1,k]$, then the restriction of
$T_{\Pi}^{k,\varepsilon}$ to $H^{s,s\gamma}(\Pi)$ defines a bounded operator
\begin{equation*}
T_{\Pi}^{k,\varepsilon}\,:\,H^{s,s\gamma}(\Pi)\rightarrow
H^{s,s\gamma}(\mathbb{R}^{2}).
\end{equation*}
\item [(iii)] Let $E$ be an open interval (bounded or not) that lies on
$\partial\Pi$, and let $\nu$ be the unit vector of an inner normal to $\partial\Pi$
(with respect to $\Pi$). If a function $v\in L_2(\Pi)$ is equal to zero on the set
$\{x_{1}+x_{2}\nu:x_1\in E,\,0<x_{2}<\varepsilon\}$, then
$T_{\Pi}^{k,\varepsilon}v\equiv0$ on the set $\{x_{1}+x_{2}\nu:x_1\in
E,\,x_{2}<0\}$.
\end{itemize}
\end{lemma}

\begin{proof}
Without loss of generality we may restrict ourselves to the case when
$\Pi=\{(x,t):\,x\in\mathbb{R},\,t>0\}$. (The general situation is reduced to this
case by translation and reflection in the plain.) We construct the operator
$T_{\Pi}^{k,\varepsilon}$ with the help of the extension method by M.~R.~Hestenes
(see \cite[Sec. 9.9]{BesovIlinNikolskii78} or \cite[Sec. 2.9.1]{Triebel95}).

Namely, given a function $v:\overline{\Pi}\to\mathbb{C}$, let
$$
(T_{\Pi}^{k,\varepsilon}v)(x,t):=
\begin{cases}
\;v(x,t)\quad&\mbox{for}\quad x\in\mathbb{R},\;t\geq 0,\\
\;\chi_{\varepsilon}(t)\sum\limits_{j=1}^{k+1}\lambda_j\,v(x,-t/j)\quad&\mbox{for}\quad
x\in\mathbb{R},\;t<0.
\end{cases}
$$
Here the numbers $\lambda_1,\ldots,\lambda_k,\lambda_{k+1}$ are chosen so that
\begin{equation*}
\sum_{j=1}^{k+1}\lambda_j\left(-\frac{1}{j}\right)^\alpha=1, \quad
\alpha=0,1,\dots,k.
\end{equation*}
Moreover, $\chi_{\varepsilon}\in C^{\infty}(\mathbb{R})$ is a fixed function such
that $\chi_{\varepsilon}(t)=1$ if $t>-\varepsilon/3$ and that
$\chi_{\varepsilon}(t)=0$ if $t<-2\varepsilon/3$. Then $v\in C^{k}(\overline{\Pi})$
implies $T_{\Pi}^{(k)}v\in C^{k}(\mathbb{R}^{2})$.

Evidently, the mapping $v\mapsto T_{\Pi}^{k,\varepsilon}v$ defines a bounded linear
operator $T_{\Pi}^{k,\varepsilon}:L_2(\Pi)\to L_2(\mathbb{R}^2)$ that complies with
conditions (i) and (iii). According to \cite[Sec. 9.9]{BesovIlinNikolskii78} this
operator satisfies condition (ii) as well.
\end{proof}

\begin{lemma}\label{lem5.3}
In addition to \eqref{f5.1} suppose that all the numbers $s_{0}$, $s_{1}$,
$s_{0}\gamma$, and $s_{1}\gamma$ are positive integers. Then
\begin{align}\label{f5.4}
H^{s,s\gamma,\varphi}_{+}(\mathbb{R}^{2}) & =
\bigl[H^{s_{0},s_{0}\gamma}_{+}(\mathbb{R}^{2}),
H^{s_{1},s_{1}\gamma}_{+}(\mathbb{R}^{2})\bigr]_{\psi},\\
H^{s,s\gamma,\varphi}_{+}(\Omega) & = \bigl[H^{s_{0},s_{0}\gamma}_{+}(\Omega),
H^{s_{1},s_{1}\gamma}_{+}(\Omega)\bigr]_{\psi} \label{f5.5}
\end{align}
with equivalence of norms.
\end{lemma}

\begin{proof}
First deduce \eqref{f5.4}. Let $\Pi:=\{(x,t):\,x\in\mathbb{R},t<0\}$, and let
$T_{\Pi}^{s_{1},1}$ be the extension operator from Lemma~\ref{lem5.2}. The mapping
$P:w\mapsto w-T^{s_{1},1}_{\Pi}(w\!\!\upharpoonright\!\Pi)$, where $w\in
L_{2}(\mathbb{R}^{2})$, defines the projector of the space
$H^{s_{j},s_{j}\gamma}(\mathbb{R}^{2})$ onto its subspace
$H^{s_{j},s_{j}\gamma}_{+}(\mathbb{R}^{2})$ for every $j\in\{0,\,1\}$. Therefore by
Proposition~\ref{prop4.3} and Lemma~\ref{lem5.1} we may write
\begin{align*}
\bigl[H^{s_{0},s_{0}\gamma}_{+}(\mathbb{R}^{2}),
H^{s_{1},s_{1}\gamma}_{+}(\mathbb{R}^{2})\bigr]_{\psi}&=
\bigl[H^{s_{0},s_{0}\gamma}(\mathbb{R}^{2}),
H^{s_{1},s_{1}\gamma}(\mathbb{R}^{2})\bigr]_{\psi}\cap
H^{s_{0},s_{0}\gamma}_{+}(\mathbb{R}^{2})\\
&=H^{s,s\gamma,\varphi}(\mathbb{R}^{2})\cap
H^{s_{0},s_{0}\gamma}_{+}(\mathbb{R}^{2})= H^{s,s\gamma,\varphi}_{+}(\mathbb{R}^{2})
\end{align*}
up to equivalence of norms. Formula \eqref{f5.4} is proved.

Now we will deduce \eqref{f5.5} from \eqref{f5.4}. To this end we construct a
certain projector $P_{0}$ of each space $H^{s_{j},s_{j}\gamma}_{+}(\mathbb{R}^{2})$,
with $j\in\{0,1\}$, onto its subspace $H^{s_{j},s_{j}\gamma}_{Q}(\mathbb{R}^{2})$
defined by \eqref{f3.8}. Consider the half-plains
\begin{gather*}
\Pi_1:=\{(x,t)\,:\,x\in\mathbb{R},\,t<\tau\},\\
\Pi_2:=\{(x,t)\,:\,x<l,\,t\in\mathbb{R}\},\\
\Pi_3:=\{(x,t)\,:\,x>0,\,t\in\mathbb{R}\}.
\end{gather*}
For every $\alpha\in\{1,2,3\}$, let $R_{\alpha}$ denote the restriction mapping
$w\mapsto w\!\upharpoonright\!\Pi_{\alpha}$, with $w\in L_{2}(\mathbb{R}^{2})$, and
let $T_{\alpha}$ denote the extension operator $T_{\Pi_{\alpha}}^{s_{1},l}$ from
Lemma~\ref{lem5.2}. Consider the mapping $P_{0}:\,w\mapsto w-\Lambda w$, with $w\in
H^{s_{0},s_{0}\gamma}_{+}(\mathbb{R}^{2})$ and $\Lambda
w:=T_{3}R_{3}T_{2}R_{2}T_{1}R_{1}w$. It follows from lemma~\ref{lem5.2} that $P_{0}$
is the projector required. Indeed, $P_{0}$ is a linear bounded operator on
$H^{s_{j},s_{j}\gamma}_{+}(\mathbb{R}^{2})$ for every $j\in\{0,1\}$. Moreover, if
$w=0$ in $\Omega$, then $\Lambda w=0$ in $\mathbb{R}^{2}$; therefore $P_{0}w=w$ for
each $w\in H^{s_{j},s_{j}\gamma}_{Q}(\mathbb{R}^{2})$.

Since the projector $P_{0}$ is given, we may apply Proposition~\ref{prop4.3} and
formula \eqref{f5.4} and write
\begin{align*}
&\bigl[H^{s_{0},s_{0}\gamma}_{+}(\Omega),
H^{s_{1},s_{1}\gamma}_{+}(\Omega)\bigr]_{\psi}]
\\
& \quad =\bigl[H^{s_{0},s_{0}\gamma}_{+}(\mathbb{R}^{2})/
H^{s_{0},s_{0}\gamma}_{Q}(\mathbb{R}^{2}),
H^{s_{1},s_{1}\gamma}_{+}(\mathbb{R}^{2})/
H^{s_{1},s_{1}\gamma}_{Q}(\mathbb{R}^{2})\bigr]_{\psi}
\\
& \quad =\bigl[H^{s_{0},s_{0}\gamma}_{+}(\mathbb{R}^{2}),
H^{s_{1},s_{1}\gamma}_{+}(\mathbb{R}^{2})\bigr]_{\psi}\big/
\bigl(\bigl[H^{s_{0},s_{0}\gamma}_{+}(\mathbb{R}^{2}),
H^{s_{1},s_{1}\gamma}_{+}(\mathbb{R}^{2})\bigr]_{\psi}\cap
H^{s_{0},s_{0}\gamma}_{Q}(\mathbb{R}^{2})\bigr)
\\
& \quad =H^{s,s\gamma,\varphi}_{+}(\mathbb{R}^{2})/
\bigl(H^{s,s\gamma,\varphi}_{+}(\mathbb{R}^{2})\cap
H^{s_{0},s_{0}\gamma}_{Q}(\mathbb{R}^{2})\bigr)=
H^{s,s\gamma,\varphi}_{+}(\mathbb{R}^{2})/H^{s,s\gamma,\varphi}_{Q}(\mathbb{R}^{2})
\\
& \quad =H^{s,s\gamma,\varphi}_{+}(\Omega)
\end{align*}
up to equivalence of norms. Formula \eqref{f5.5} is proved.
\end{proof}

It remains to prove a necessary interpolation formula for the space
$H^{s,\varphi}_{+}(0,\tau)$.

\begin{lemma}\label{lem5.4}
In addition to \eqref{f5.1} suppose that $s_{0}\geq0$. Then
\begin{equation}\label{f5.6}
H^{s,\varphi}_{+}(0,\tau)=
\bigl[H^{s_{0}}_{+}(0,\tau),H^{s_{1}}_{+}(0,\tau)\bigr]_{\psi}
\end{equation}
with equivalence of norms.
\end{lemma}

\begin{proof}
The formula \eqref{f5.6} can be deduced by analogy with the anisotropic spaces case
considered in the previous lemmas. For the sake of the argumentation completeness,
let us give the proof.

First note that an analog of Lemma~\ref{lem5.1} for isotropic spaces over
$\mathbb{R}^{n}$ is proved in \cite[Sec. 3.2, Theorem 3.4]{08MFAT1} (also see
\cite[Sec. 1.3.4, Theorem 1.14]{MikhailetsMurach10}). Specifically,
\begin{equation}\label{f5.7}
H^{s,\varphi}(\mathbb{R})=
\bigl[H^{s_{0}}(\mathbb{R}),H^{s_{1}}(\mathbb{R})\bigr]_{\psi}
\end{equation}
with equality of norms.

To deduce \eqref{f5.6} from \eqref{f5.7} we will apply the following one-dimensional
analog of Lemma~\ref{lem5.2} on extension operator. Let $G\subset\mathbb{R}$ be an
open semiaxis and $k\in\mathbb{N}$. Then there exists a bounded linear operator
$T^{(k)}_{G}:L_{2}(G)\rightarrow L_{2}(\mathbb{R})$ such that $T^{(k)}_{G}v$ is an
extension of $v\in L_{2}(G)$ and that the mapping $v\mapsto T^{(k)}_{G}v$ defines
the bounded operator $T^{(k)}_{G}:\nobreak H^{s}(G)\rightarrow H^{s}(\mathbb{R})$
for every real $s\in[0,k)$. Here, as usual,
\begin{gather*}
H^{s}(G):=\{h\!\upharpoonright\!G:h\in H^{s}(\mathbb{R})\},\quad \mbox{with}   \\
\|v\|_{H^{s}(G)}:= \inf\bigl\{\|h\|_{H^{s}(\mathbb{R})}:\,h\in
H^{s}(\mathbb{R}),\;\;h=v\;\,\mbox{in}\;\,G\bigr\},
\end{gather*}
is the Sobolev space over $G$ of order $s$. This analog is a special case of Lemma
2.9.3 from \cite{Triebel95}. As above, the operator $T^{(k)}_{G}$ can be constructed
with the help of the extension method by M.~R.~Hestenes.

Chose $k\in\mathbb{N}$ so that $s_{1}<k$. The mapping $P:h\mapsto
h-T^{(k)}_{G}(h\!\!\upharpoonright\!G)$, where $h\in L_{2}(\mathbb{R})$ and
$G:=(-\infty,0)$, defines the projector of the space $H^{s_{j}}(\mathbb{R})$ onto
its subspace $H^{s_{j}}_{+}(\mathbb{R})$ for every $j\in\{0,\,1\}$. Therefore by
Proposition~\ref{prop4.3} and formula \eqref{f5.7} we may write
\begin{equation}\label{f5.8}
\bigl[H^{s_{0}}_{+}(\mathbb{R}), H^{s_{1}}_{+}(\mathbb{R})\bigr]_{\psi}=
\bigl[H^{s_{0}}(\mathbb{R}),H^{s_{1}}(\mathbb{R})\bigr]_{\psi}\cap
H^{s_{0}}_{+}(\mathbb{R})=H^{s,\varphi}_{+}(\mathbb{R})
\end{equation}
up to equivalence of norms.

Now let us deduce \eqref{f5.6} from \eqref{f5.8}. Recall that
$H^{s,\varphi}_{+}(0,\tau)$ is the factor space of the space
$H^{s,\varphi}_{+}(\mathbb{R})$ by its subspace \eqref{f3.11}. The latter coincides
with
$$
H^{s,\varphi}_{[\tau,\infty)}(\mathbb{R}):=\bigl\{h\in H^{s,\varphi}(\mathbb{R}):\,
\mathrm{supp}\,h\subseteq[\tau,\infty)\bigr\}
$$
because $s>0$. The mapping $P_{\tau}:h\mapsto
h-T^{(k)}_{G_{\tau}}(h\!\!\upharpoonright\!G_{\tau})$, where $h\in
L_{2}(\mathbb{R})$ and $G_{\tau}:=(-\infty,\tau)$, sets the projector of the space
$H^{s_{j}}_{+}(\mathbb{R})$ onto its subspace
$H^{s_{j}}_{[\tau,\infty)}(\mathbb{R})$ for every $j\in\{0,\,1\}$. Therefore by
Proposition~\ref{prop4.3} and formula \eqref{f5.8} we may write
\begin{align*}
\bigl[H^{s_{0}}_{+}(0,\tau), H^{s_{1}}_{+}(0,\tau)\bigr]_{\psi}&=
\bigl[H^{s_{0}}_{+}(\mathbb{R})/H^{s_{0}}_{[\tau,\infty)}(\mathbb{R}),
H^{s_{1}}_{+}(\mathbb{R})/
H^{s_{1}}_{[\tau,\infty)}(\mathbb{R})\bigr]_{\psi}\\
&=\bigl[H^{s_{0}}_{+}(\mathbb{R}),H^{s_{1}}_{+}(\mathbb{R})\bigr]_{\psi}\big/
\bigl(\bigl[H^{s_{0}}_{+}(\mathbb{R}),H^{s_{1}}_{+}(\mathbb{R})\bigr]_{\psi}\cap
H^{s_{0}}_{[\tau,\infty)}(\mathbb{R})\bigr)\\
&=H^{s,\varphi}_{+}(\mathbb{R})/H^{s,\varphi}_{[\tau,\infty)}(\mathbb{R})=
H^{s,\varphi}_{+}(0,\tau)
\end{align*}
up to equivalence of norms. Formula \eqref{f5.6} is proved.
\end{proof}

Now we may give

\begin{proof}[The proof of Main Theorem.]
Let $\sigma>\sigma_0$ and $\varphi\in\mathcal{M}$. Chose a number
$\sigma_1\in\mathbb{N}$ so that $\sigma_1/(2b)\in\mathbb{N}$ and $\sigma_1>\sigma$.
According to M.~S.~Agranovich and M.~I.~Vishik \cite[Theorem~11.1]{AgranovichVishik64},
the mapping \eqref{f2.5} extends uniquely to isomorphisms
between Sobolev spaces
\begin{equation}\label{f5.9}
(A,B):\,H^{\sigma_k,\sigma_k/(2b)}_{+}(\Omega)\leftrightarrow\mathcal{H}_k
\quad\mbox{for every}\quad k\in\{0,1\},
\end{equation}
where
$$
\mathcal{H}_k:=H^{\sigma_k-2m,(\sigma_k-2m)/(2b)}_{+}(\Omega)\oplus
\bigoplus_{j=1}^{m}\bigl(H^{(\sigma_k-m_j-1/2)/(2b)}_{+}(0,\tau)\bigr)^{2}.
$$

Define an interpolation parameter by the formula
\begin{equation*}
\psi(r):=
\begin{cases}
\;r^{(\sigma-\sigma_{0})/(\sigma_{1}-\sigma_{0})}\,
\varphi(r^{1/(\sigma_{1}-\sigma_{0})})&\text{for}\quad r\geq1, \\
\;\varphi(1) & \text{for}\quad0<r<1,
\end{cases}
\end{equation*}
which is analogous to \eqref{f5.2}. Applying the interpolation with the function
parameter $\psi$ to \eqref{f5.9}, we get another isomorphism
\begin{equation}\label{f5.10}
(A,B):\,\bigl[H^{\sigma_0,\sigma_0/(2b)}_{+}(\Omega),
H^{\sigma_1,\sigma_1/(2b)}_{+}(\Omega)\bigr]_{\psi}\leftrightarrow
[\mathcal{H}_0,\mathcal{H}_1]_{\psi}.
\end{equation}
This isomorphism is a unique extension by continuity of the mapping \eqref{f2.5}
because $C^{\infty}_{+}(\overline{\Omega})$ is dense in the domain of \eqref{f5.10}.

Let us describe the interpolation spaces appearing in
\eqref{f5.10}. According to Lemma \ref{lem5.3} we have
\begin{equation*}
\bigl[H^{\sigma_0,\sigma_0/(2b)}_{+}(\Omega),
H^{\sigma_1,\sigma_1/(2b)}_{+}(\Omega)\bigr]_{\psi}=
H^{\sigma,\sigma/(2b),\varphi}_{+}(\Omega)
\end{equation*}
with equality of norms. Next, applying Proposition~\ref{prop4.4} and Lemmas
\ref{lem5.3} and \ref{lem5.4} we may write
\begin{align*}
[\mathcal{H}_0,\mathcal{H}_1]_{\psi}&=
\bigl[H^{\sigma_0-2m,(\sigma_0-2m)/(2b)}_{+}(\Omega),
H^{\sigma_1-2m,(\sigma_1-2m)/(2b)}_{+}(\Omega)\bigr]_{\psi}\\
&\oplus\bigoplus_{j=1}^{m}\bigl(\bigl[H^{(\sigma_0-m_j-1/2)/(2b)}_{+}(0,\tau),
H^{(\sigma_1-m_j-1/2)/(2b)}_{+}(0,\tau)\bigr]_{\psi}\bigr)^{2}\\
&=H^{\sigma-2m,(\sigma-2m)/(2b),\varphi}_{+}(\Omega)\oplus
\bigoplus_{j=1}^{m}\bigl(H^{(\sigma-m_j-1/2)/(2b),\varphi}_{+}(0,\tau)\bigr)^{2}
\end{align*}
with equality of norms. Note that the function $\psi$ satisfies \eqref{f5.2} because
the parameters $s_{0}$, $s_{1}$, and $s$ in these lemmas differ from $\sigma_{0}$,
$\sigma_{1}$, and $\sigma$ respectively in the same magnitude. Thus, the isomorphism
\eqref{f5.10} becomes \eqref{f2.6}.
\end{proof}

\section{Final remarks}\label{sec6}

Main Theorem can be used to investigate regularity of solutions to parabolic
problems. Specifically, applying H\"ormander's Embedding Theorem \cite[Theorem
2.2.7]{Hormander63}, we may establish sufficient conditions for the weak solution to
be classical (compare with \cite[Sec. 5 and 6]{07UMJ5} or \cite[Sec.
4.1.2]{MikhailetsMurach10}, where elliptic boundary--value problems are considered).

The investigation of parabolic initial--boundary value problems with nonhomogeneous
initial conditions can be reduced to the case of homogeneous ones (see
\cite[\S~10]{AgranovichVishik64} in the case of Sobolev spaces). In this connection,
we also mention J.-L.~Lions and E.~Magenes' approach \cite[Sec.
6.4]{LionsMagenes72ii} based on interpolation with a number parameter. Apparently,
their methods may admit a generalization to the case of function interpolation
parameters.

An analog of Main Theorem is also true for the many-dimensional case, when the
parabolic problem is given in a cylinder situated in $\mathbb{R}^{n+1}$, with
$n\geq2$. This analog can be deduced from M.~S.~Agranovich and M.~I.~Vishik's result
\cite[Theorem~11.1]{AgranovichVishik64} by means of interpolation with a function
parameter.

The above-mentioned applications and generalizations of Main Theorem will be
published elsewhere.


\end{document}